\documentclass[11pt, a4paper, conference]{IEEEtran} 

\IEEEoverridecommandlockouts                              

\usepackage{graphicx}
\usepackage{subcaption}
\usepackage{amsmath}

\usepackage{amsthm}
\usepackage{dsfont}
\usepackage{xcolor}
\usepackage{upgreek}
\usepackage{enumerate}
\usepackage{amsfonts}
\usepackage{amssymb}
\usepackage{textcomp}
\usepackage{tikz}
\usepackage{tkz-berge}
\usetikzlibrary{calc}
\usetikzlibrary{shapes,fit}
\usepackage{verbatim}
\usepackage{setspace}
\usetikzlibrary{shapes,arrows}

\newtheorem{theorem}{Theorem}[section]
\newtheorem{property}[theorem]{Property}
\newtheorem{lemma}[theorem]{Lemma}
\newtheorem{assumption}[theorem]{Assumption}

\newtheorem{corollary}[theorem]{Corollary}

\newtheorem{example}[theorem]{Example}
\newtheorem{remark}[theorem]{Remark}

\definecolor{cof}{RGB}{219,144,71}
\definecolor{pur}{RGB}{186,146,162}
\definecolor{greeo}{RGB}{91,173,69}
\definecolor{greet}{RGB}{52,111,72}

\usepackage[normalem]{ulem} 
\usepackage{soul}
\usepackage{cancel}

\newcommand{\bal}[1] {\ensuremath{\left(\begin{array}{#1}}}
\newcommand{\ear} {\ensuremath{\end{array}\right)}}

\newcommand{\bals}[1] {\ensuremath{\left[\begin{array}{#1}}} 
\newcommand{\ears} {\ensuremath{\end{array} \right] }} 
\newcommand{\co} {\ensuremath{\mathrm{co}}} 

\DeclareMathOperator{\spa}{span}

\DeclareMathOperator{\rank}{rank}

\DeclareMathOperator{\sign}{sign}

\DeclareMathOperator{\diag}{diag}

\newcommand{\funcRdR}{\ensuremath{{f}}}
\newcommand{\funcRdRd}{\ensuremath{{X}}}

\DeclareMathOperator*{\bigtimes}{\raisebox{-0.3ex}{\text{\Large$\times$}}}

\let\leq\leqslant
\let\geq\geqslant
\let\emptyset\varnothing

\newcommand{\calC}{\ensuremath{\mathcal{C}}}
\newcommand{\calD}{\ensuremath{\mathcal{D}}}
\newcommand{\calE}{\ensuremath{\mathcal{E}}}
\newcommand{\calF}{\ensuremath{\mathcal{F}}}
\newcommand{\calG}{\ensuremath{\mathcal{G}}}
\newcommand{\calH}{\ensuremath{\mathcal{H}}}
\newcommand{\calI}{\ensuremath{\mathcal{I}}}

\newcommand{\calK}{\ensuremath{\mathcal{K}}}
\newcommand{\calL}{\ensuremath{\mathcal{L}}}

\newcommand{\calN}{\ensuremath{\mathcal{N}}}

\newcommand{\calR}{\ensuremath{\mathcal{R}}}

\newcommand{\calV}{\ensuremath{\mathcal{V}}}

\newcommand{\bmat}{\begin{matrix}}
\newcommand{\emat}{\end{matrix}}
\newcommand{\bbm}{\begin{bmatrix}}
\newcommand{\ebm}{\end{bmatrix}}
\newcommand{\bpm}{\begin{pmatrix}}
\newcommand{\epm}{\end{pmatrix}}
\newcommand{\bse}{\begin{subequations}}
\newcommand{\ese}{\end{subequations}}
\newcommand{\beq}{\begin{equation}}
\newcommand{\eeq}{\end{equation}}
\newcommand{\ben}{\begin{enumerate}}
\newcommand{\een}{\end{enumerate}}
\newcommand{\beni}{\renewcommand{\labelenumi}{\roman{enumi}.}
\renewcommand{\theenumi}{\roman{enumi}}\begin{enumerate}}
\newcommand{\eeni}{\end{enumerate}\renewcommand{\labelenumi}{\arabic{enumi}.}
\renewcommand{\theenumi}{\arabic{enumi}}}
\newcommand{\bena}{\renewcommand{\labelenumi}{\alpha{enumi}.}
\renewcommand{\theenumi}{\alpha{enumi}}\begin{enumerate}}
\newcommand{\eena}{\end{enumerate}\renewcommand{\labelenumi}{\arabic{enumi}.}
\renewcommand{\theenumi}{\arabic{enumi}}}
\newcommand{\bit}{\begin{itemize}}
\newcommand{\eit}{\end{itemize}}

\newcommand{\R}{\ensuremath{\mathbb R}}

\newcommand{\Z}{\ensuremath{\mathbb Z}}

\newcounter{foochapter} 

\newcommand{\addchapterfront}{
\setcounter{foochapter}{\the\value{chapter}}
\stepcounter{foochapter}
\cleardoublepage
\thispagestyle{empty}
\vspace*{40mm}
\begin{flushright}
\fontfamily{ugm}\textrm{\huge Chapter~}
\color{RoyalRed}\fontfamily{ugm}\fontsize{70}{70}\selectfont\the\value{foochapter} \\
\vspace{2mm}
\color{black}\hrule
\vspace{2mm}
\includegraphics[height = 20.3mm]{./front/bw_logo_rug}
\end{flushright}
\renewcommand{\thepage}{[\arabic{page}]}
\cleardoublepage
\addtocounter{page}{-2}
\renewcommand{\thepage}{\arabic{page}}
}

\title{\LARGE \bf
Nonlinear consensus protocols with applications to quantized systems}

\author{Jieqiang Wei, Xinlei Yi, Henrik Sandberg and Karl Henrik Johansson 
\thanks{
 J. Q. Wei, X. L. Yi, H. Sandberg and K. H. Johanson are with the ACCESS Linnaeus Centre, Electrical Engineering, KTH Royal Institute of Technology, 100 44, Stockholm, Sweden,
        {\tt\small xinleiy@kth.se, jieqiang.wei@gmail.com, hsan@ee.kth.se, kallej@kth.se}.
}}

\begin{document}

\maketitle
\thispagestyle{plain}
\pagestyle{plain}

\begin{abstract}\label{s:Abstract}
Two types of general nonlinear consensus protocols are considered in this paper, namely the systems with nonlinear measurement and communication of the agents' states, respectively.  The solutions of the systems are understood in the sense of Filippov to handle the possible discontinuity of the nonlinear functions. For each case, we prove the asymptotic stability of the systems defined on both directed and undirected graphs. Then we reinterpret the results about the general models for a specific type of systems, i.e.,  the quantized consensus protocols, which extend some existing results (e.g., \cite{Ceragioli2011}, \cite{Guo2013}) from undirected graphs to directed ones.
\end{abstract}

\section{Introduction}\label{s:Introduction}

Apart from the popular linear consensus protocols, nonlinear agreement protocols have recently attracted the attention of many researchers. As a special type of nonlinear consensus protocols, quantized consensus protocols have been studied from different viewpoints. In fact, quantization can be due to digital communication, to coarse sensing capabilities, and/or to limited precision in computation.

Some related works about the quantized systems are as follows. Generally speaking, there are two major divisions about the quantized systems. The first one is that the measurement of the states is quantized, see e.g., \cite{Ceragioli2011}, \cite{Frasca2012}, \cite{Chen2013}, \cite{Ceragioli2015}. In particular, the results in \cite{Ceragioli2011} and \cite{Frasca2012} are the most related to the current paper, where the authors considered the consensus protocols with quantized states measurement on \emph{undirected} graphs.
The other one is that the communications among the agents are quantized, see e.g., \cite{Guo2013}, \cite{Dimos2010} and \cite{Kashyap2007}. In \cite{Dimos2010}, the authors considered quantized communication protocols within the framework of hybrid dynamical systems. In \cite{Guo2013}, the authors considered the communication quantized system using the notions of Filippov solutions for \emph{undirected} graphs. In \cite{YiCDC2016}, the authors considered both divisions and proposed self-triggered rules to avoid continuous communications between agents.

Another major motivation of this paper is \cite{Wei2015} where the authors considered several nonlinear consensus protocols with the fundamental assumptions of the nonlinear functions being sign-preserving, i.e., the function takes strictly positive values for positive variables and vice versa. However this property is not satisfied by some quantizers. This motivates us to consider a framework of nonlinear consensus protocols without sign-preserving but only with monotone assumption of the nonlinear functions.

The contributions of this paper are twofolds. First, we present the stability of two general nonlinear consensus protocols, namely the protocols with nonlinear measurement and communication of the states, for all of the Filippov solutions. In these models, one fundamental assumption is the monotonicity of these nonlinear functions. In addition, some extra conditions are needed in order to guarantee the boundedness of all the Filippov trajectories. Second, we reinterpret the results about general systems to a special case, i.e., quantized consensus protocols, which serves as an extension of the results in \cite{Ceragioli2011}, \cite{Guo2013} from undirected graphs to directed ones.

The structure of the paper is as follows. In Section \ref{s:Preliminaries}, we introduce some terminologies, notations and lemmas. 
In Section \ref{s:measure}, we consider the nonlinear consensus protocols where the measurement of the state of the agents are effected by some nonlinearities. 
Section \ref{s:communication} is devoted to the case when the communication among the agents are imprecise. 
In Section \ref{s:quantize} we reinterpret the results in Section \ref{s:measure} and \ref{s:communication} for the quantized consensus protocols.
Finally, the conclusion follows.


\section{Preliminaries}\label{s:Preliminaries}
In this section we briefly review some notions from graph
theory, and give some definitions, notations and properties regarding Filippov solutions.

Let $\mathcal{G}=(\mathcal{V},\mathcal{E},A)$ be a weighted digraph with
node set $\mathcal{V}=\{v_1,\ldots,v_n\}$,
edge set $\mathcal{E}\subseteq\mathcal{V}\times\mathcal{V}$ and
weighted adjacency matrix $A=[a_{ij}]$ with nonnegative adjacency elements $a_{ij}$.
An edge of $\mathcal{G}$ is denoted by $e_{ij}=(v_i,v_j)$ and we write $\calI=\{1,2,\ldots,n\}$.
The adjacency elements $a_{ij}$ are associated with the edges of the graph in
the following way: $a_{ij}>0$  if and only if  $e_{ji} \in \mathcal{E}$.
Moreover,  $a_{ii}=0$ for all $i \in\calI$.
For undirected graphs,  $A=A^T$.

The set of neighbors of node $v_i$ is
denoted by $N_i=\{v_j\in\mathcal{V}:(v_j,v_i)\in\mathcal{E}\}$.
For each node $v_i$, its in-degree
is defined as
\begin{align*}
\deg_{\rm in} (v_i) = \sum_{j=1}^n a_{ij},  
\end{align*}
The degree
matrix of the digraph $\mathcal{G}$ is a diagonal matrix $\Delta$
where $\Delta_{ii}=\deg_{\rm  in}(v_i)$. The \emph{graph Laplacian} is
defined as
\begin{equation*}
L=\Delta-A.
\end{equation*}
This implies $L\mathds{1}_n =0_n$, where $\mathds{1}_n$ is the $n$-vector containing only ones and $0_n$ is the $n$-vector containing only zeros.

A directed path from node $v_i$ to node $v_j$ is a chain of edges from $\calE$
such that the first edge starts from $v_i$, the last edge ends at  $v_j$ and
every edge starts where the previous edge ends.
A graph is called \emph{strongly connected} if for every two nodes $v_i$ and
$v_j$ there is a directed path from $v_i$ to $v_j$.
A subgraph $\calG' = (\calV',\calE',A')$ of $\calG$ is called a \emph{directed
	spanning tree} for $\calG$ if $\calV' =\calV $, $\calE' \subseteq \calE$, and for every node $v_i\in
\calV'$ there is exactly one $v_j$ such that $e_{ji}\in \calE'$, except for one
node, which is called the root of the spanning tree.
Furthermore, we call a node $v\in \calV$ a \emph{root} of $\calG$ if there is a directed
spanning tree for $\calG$ with $v$ as a root.
In other words, if $v$ is a root of $\calG$, then there is a directed path from
$v$ to every other node in the graph. A digraph is {\it a directed ring} if for every node $v_i$, there exists exactly one $v_j$ such that $e_{ij}\in\calE$ and there exists exactly one $v_k$ such that $e_{ki}\in\calE$.

A digraph, with $m$ edges, is completely specified by its \emph{incidence
	matrix} $B$, which is an $n\times m$ matrix, with $(i,j)^{\text{th}}$
element equal to $-1$ if the $j^{\text{th}}$ edge is towards vertex
$i$, and equal to $1$ if the $j^{\text{th}}$ edge is originating from
vertex $i$, and $0$ otherwise.

An important property about strong connected digraph is
\begin{property} [Lemma 2 in \cite{Lu2007}] \label{connected_laplacian}
	The graph Laplacian matrix $L$ of a strongly connected digraph
	$\calG$ satisfies: zero is an algebraically simple
eigenvalue of $L$ and there is a positive vector
$w^{\top}=[w_{1},\cdots,w_{n}]$ such that $w^{\top} L=0$ and
$\sum_{i=1}^{m}w_{i}=1$. Moreover $L^{\top}diag(w)$ is positive semi-definite.
\end{property}

With $\mathbb{R}_-$, $\mathbb{R}_+$ and $\R_{\geqslant 0}$ we denote the sets of
negative, positive and nonnegative real numbers, respectively. The $i$th row and $j$th column of a matrix $M$ are denoted as $M_{i,\cdot}$ and $M_{\cdot,j}$, respectively. And for simplicity, let $M_{\cdot,j}^{\top}$ denotes $(M_{\cdot,j})^\top$.

The vectors $e_1,e_2,\ldots,e_n$ denote the canonical basis of $\R^n$.

In the rest of this section we give some definitions and notations regarding
Filippov solutions (see, e.g., \cite{cortes2008}).

Let $\funcRdRd$ be a map from $\R^n$ to $\R^n$,
and let $2^{\R^n}$ denotes the collection of all subsets of $\R^n$.
We define the \emph{Filippov set-valued map} of $\funcRdRd$, denoted
$\calF[\funcRdRd]:\R^n\rightarrow 2^{\R^n}$,  as
\begin{equation}\label{e:def_filippovset}
\calF[\funcRdRd](x)\triangleq
\bigcap_{\delta>0}\bigcap_{\mu(S)=0}\overline{\mathrm{co}}\{\funcRdRd(B(x,
\delta)\backslash S)\},
\end{equation}
where $B(x,\delta)$ is the open ball centered at $x$ with radius $\delta>0$,
$S$ is a subset of $\R^n$,
$\mu$ denotes the Lebesgue measure and $\overline{\mathrm{co}}$ denotes the convex closure.
If $X$ is continuous at $x$, then $ \calF[\funcRdRd](x)$ contains only the point $X(x)$.
Moreover, there are some useful properties about the Filippov set-valued map.

\begin{property}[Calculus for $\calF$ \cite{paden1987}]\label{p:calculus for Filippov}
	\begin{enumerate}[(i)]
		\item Assume that $f:\R^m\rightarrow\R^n$ is locally bounded. Then $\exists N_f\subset \R^m, \mu(N_f)=0$ such that $\forall N\subset\R^m, \mu(N)=0$,
		\begin{equation}
		\calF[f](x)=\mathrm{co}\{\lim_{i\rightarrow\infty} f(x_i)\mid x_i\rightarrow x, x_i\notin N_f\cup N \}.
		\end{equation}
		\item Assume that $f_j:\R^m\rightarrow \R^{n_j}, j=1,\ldots,N$ are locally bounded, then
		\begin{equation}
		\calF\big[ \bigtimes_{j=1}^N f_j \big](x) \subset \bigtimes_{j=1}^N\calF[f_j](x).\footnote{Cartesian product notation and column vector notation are used interchangeably.}
		\end{equation}
		\item Let $g:\R^m\rightarrow\R^n$ be $C^1$, $\rank Dg(x)=n$, where $Dg(x)$ is the Jacobian matrix, and $f:\R^n\rightarrow\R^p$ be locally bounded; then
		\begin{equation}
		\calF[f\circ g](x)=\calF[f](g(x)).
		\end{equation}
		\item Let $g:\R^m\rightarrow\R^{p\times n}$ (i.e. matrix valued) be $C^0$ and $f:\R^m\rightarrow\R^n$ be locally bounded; then
		\begin{equation}
		\calF[gf](x)=g(x)\calF[f](x)
		\end{equation}
		where $gf(x):=g(x)f(x)\in\R^p$.
	\end{enumerate}
\end{property}

\begin{property}\label{p:fili_monotone}
	For an increasing function $\varphi:\R\rightarrow\R$, the Filippov set-valued map satisfies that
	\begin{enumerate}[(i)]
		\item  $\calF[\varphi](x)=[\varphi(x^-),\varphi(x^+)]$ where $\varphi(x^-),\varphi(x^+)$ are the left and right limit of $\varphi$ at $x$ respectively;
		\item  for any $x_1<x_2$, and $\nu_i\in\calF[\varphi](x_i), i=1,2,$ we have $\nu_1\leq\nu_2$.
	\end{enumerate}
\end{property}

\begin{proof}
	This can be seen as a straightforward deduction from Property \ref{p:calculus for Filippov} (1) and the definition of increasing functions.
\end{proof}

By using the fact that monotone functions are continuous almost everywhere, and the definition of right and left limits, we have following property.

\begin{property}\label{p:monotone}
	For an increasing function $\varphi:\R\rightarrow\R$,
	\begin{enumerate}[(i)]
		\item  $\calF[\varphi](x)=\{\varphi(x) \}$ for almost all $x$;
		\item  the right (left) limit, i.e., $\varphi(x^+)$ ($\varphi(x^-)$) is right (left) continuous for all $x$.
	\end{enumerate}
\end{property}

A \emph{Filippov solution} of the differential equation $\dot{x}(t)=\funcRdRd(x(t))$ on $[0,t_1]\subset\R$ is
an absolutely continuous function $x:[0,t_1]\rightarrow\R^n$ that
satisfies the differential inclusion
\begin{equation}\label{e:differential_inclusion}
\dot{x}(t)\in \calF[\funcRdRd](x(t))
\end{equation}
for almost all $t\in[0,t_1]$.
A Filippov solution $t\mapsto x(t)$ is \emph{maximal} if it cannot be extended forward in time, that is, if $t\rightarrow x(t)$ is not the result of the truncation of another solution with a larger interval of definition. Since the Filippov solutions of a discontinuous system \eqref{e:differential_inclusion} are not necessarily unique, we need to specify two types of invariant set. A set $\calR\subset\R^n$ is called \emph{weakly invariant} for \eqref{e:differential_inclusion} if, for each $x_0\in \calR$, at least one maximal solution of \eqref{e:differential_inclusion} with initial condition $x_0$ is contained in $\calR$. Similarly, $\calR\subset \R^n$ is called \emph{strongly invariant} for \eqref{e:differential_inclusion} if, for each $x_0\in \calR$, every maximal solution of \eqref{e:differential_inclusion} with initial condition $x_0$ is contained in $\calR$. For more details, see \cite{cortes2008,filippov1988}.

Let $\funcRdR$ be a map from $\R^n$ to $\R$. The right directional derivative
of $\funcRdR$ at $x$ in the direction of $v\in \R^n$
is defined as
\begin{equation*}
\funcRdR'(x;v)=\lim_{h\rightarrow 0^+} \frac{\funcRdR(x+hv)-\funcRdR(x)}{h},
\end{equation*}
when this limit exists. The generalized derivative of $\funcRdR$ at $x$ in the
direction of $v\in \R^n$ is given by
\begin{equation*}
\begin{aligned}
\funcRdR^o(x;v) & =\limsup_{\begin{subarray}{c}
	y\rightarrow x  \\
	h\rightarrow 0^+
	\end{subarray}}
\frac{\funcRdR(y+hv)-\funcRdR(y)}{h} \\
& = \lim_{\begin{subarray}{c}
	\delta \rightarrow 0^+ \\
	\epsilon \rightarrow 0^+
	\end{subarray}}
\sup_{\begin{subarray}{c}
	y\in B(x,\delta)\\
	h\in[0,\epsilon)
	\end{subarray}}
\frac{\funcRdR(y+hv)-\funcRdR(y)}{h}.
\end{aligned}
\end{equation*}
We call the function $\funcRdR$ \emph{regular} at $x$ if $ \funcRdR'(x;v)$ and
$\funcRdR^o(x;v)$ are equal for all $v \in \R^n$. For example, convex function is regular (see e.g.,\cite{Clarke1990optimization}).

If $\funcRdR : \R^n \rightarrow \R$ is locally Lipschitz, then its {\it generalized gradient}
$\partial
\funcRdR:\R^n\rightarrow 2^{\R^n}$ is defined by
\begin{equation}
\partial \funcRdR(x):=\mathrm{co}\{\lim_{i\rightarrow\infty} \nabla
\funcRdR(x_i):x_i\rightarrow x, x_i\notin S\cup \Omega_{\funcRdR} \},
\end{equation}
where $\nabla$ denotes the gradient operator, $\Omega_{\funcRdR} \subset\R^n$ denotes the set of points where
$\funcRdR$ fails to
be differentiable and $S\subset\R^n$ is a set of Lebesgue measure zero that can be
arbitrarily
chosen to simplify the computation. The resulting set $\partial \funcRdR(x)$ is independent of the choice of $S$ \cite{Clarke1990optimization}.

Given a  set-valued map $\calF:\R^n\rightarrow
2^{\R^n}$, the \emph{set-valued Lie derivative}
$\tilde{\mathcal{L}}_{\calF}\funcRdR:\R^n\rightarrow 2^{\R}$
of a locally Lipschitz function $\funcRdR:\R^n\rightarrow \R$  with respect to
$\calF$ at $x$ is
defined as
\begin{equation}\label{e:set-valuedLie}
\begin{aligned}
\tilde{\mathcal{L}}_{\calF}\funcRdR(x) := & \{a\in\R  \mid \textnormal{there
	exists } \nu\in\calF(x) \textnormal{ such that } \\
& \zeta^T\nu=a
\textnormal{
	for all } \zeta\in \partial \funcRdR(x)\}.
\end{aligned}
\end{equation}
If $\calF$ takes convex and compact values, then for each $x$, $\tilde{\mathcal{L}}_{\calF}\funcRdR(x)$ is closed and bounded interval in $\R$, possibly empty.

The following result is a generalization of LaSalle's invariance principle for discontinuous differential equations \eqref{e:differential_inclusion} with non-smooth
Lyapunov functions.
\begin{theorem}[LaSalle Invariance Principle \cite{cortes2008}]\label{chap_preli:thm_stability}
Let $\funcRdR:\R^n\rightarrow\R$ be a locally Lipschitz and regular function. Let $S\subset \R^n$ be compact and strongly invariant for \eqref{e:differential_inclusion}, and assume that $\max \tilde{\calL}_{\calF[\funcRdRd]} f(y)\leq 0$ for each $y\in S$, where we define $\max\emptyset=-\infty$. Then, all solutions $x:[0,\infty)\rightarrow \R^n$ of \eqref{e:differential_inclusion} starting at $S$ converge to the largest weakly invariant set $M$ contained in
\begin{equation}
S\cap\overline{\{y\in\R^n\mid
	0\in\tilde{\calL}_{\mathcal{F}[\funcRdRd]}f(y)\}}.
\end{equation}
Moreover, if the set $M$ consists of a finite number of points, then the limit of each solution starting in $S$ exists and is an element of $M$.
\end{theorem}

At the end of this section, we list two potential Lyapunov functions.
\begin{lemma}[Prop. 2.2.6, Ex. 2.2.8, and Prop. 2.3.6 in \cite{Clarke1990optimization}]\label{regular_Lipschitz_lyapunov}
	The following functions are regular and Lipschitz continuous,
	\begin{equation}\label{e:VandW}
	V(x):= \max_{i\in\calI} x_i, \qquad W(x):= -\min_{i\in\calI} x_i.
	\end{equation}
\end{lemma}

\section{Systems with Nonlinear Measurement}\label{s:measure}
In this section we consider a network of $n$ agents with a communication topology given by a weighted directed graph $\calG=(
\calV,\calE,A)$. In this network, agent $i$ receives information from agent $j$ if and only if
there is an edge from node $v_j$ to node $v_i$ in the graph $\calG$. Unlike the linear consensus protocol where the agents can communicate with their real states, here we propose one strategy that only a nonlinear version of the states are available to the agents. More precisely, we consider the following nonlinear consensus protocol
\begin{equation}\label{e:nonlinear1}
\dot{x} = -Lf(x)
\end{equation}
where $f(x)=[f_1(x_1),\ldots,f_n(x_n)]^T$ and $f_i:\R\rightarrow\R$.
Throughout this paper, we assume the following.

\begin{assumption}\label{as:monotone}
	The function $f_i$ is an increasing function and satisfies that $\lim_{x_i\rightarrow\infty}|f_i(x_i)|=\infty$.
\end{assumption}
Note here we do \emph{not} assume any continuity of the function $f_i$, examples include sign function, quantizations etc. In order to handle the possible discontinuities, we understand the solution of \eqref{e:nonlinear1} in the Filippov sense, i.e., we consider the differential inclusion
\begin{equation}
\dot{x} \in -L\calF[f](x).
\end{equation}
By Property \ref{p:calculus for Filippov}, the previous dynamical inclusion satisfies
\begin{equation}\label{e:nonlinear1_fili}
\begin{aligned}
\dot{x} \in -L \bigtimes_{i=1}^n\calF[f_i](x_i)
 := \calK_1(x).
\end{aligned}
\end{equation}

Denote
\begin{equation}\label{e:D}\calD_1=\{ x\in\R^n \mid \exists a\in\R \textnormal{ s.t. } a\mathds{1}_n\in \bigtimes_{i=1}^n\calF[f_i](x_i)\}.
\end{equation}

\begin{property}
	For the function $f_i$ satisfies Assumption \ref{as:monotone}, the set $\calD_1$ is closed.
\end{property}

\begin{proof}
	Take any sequence $\{y^k\}\subset\R ^n$ satisfying $\lim_{k\rightarrow\infty}y^k=x$ and $y^k\in\calD_1, k=1,2,\ldots$, we shall show that $x\in\calD_1$. Without loss of generality, we can assume the sequence satisfies that $y^k_i$ converge to $x_i$ from one side, i.e., $y^k_i<x_i$ or $y^k_i>x_i$.
	
	Note that $y^k\in\calD_1$ implies that $\cap_{i=1}^n\calF[f_i](y^k_i)\neq \emptyset.$ For the case $y^k_i>x_i$, we have $f_i(y^{k-}_i)\geq f_i(x^-_i)$, $f_i(y^{k+}_i)\geq f_i(x^+_i)$ and $\lim_{k\rightarrow\infty}f_i(y^{k+}_i)= f_i(x^+_i)$ which is based on Property \ref{p:monotone} (ii). Hence we have $[\lim_{k\rightarrow\infty}f_i(y^{k-}_i),\lim_{k\rightarrow\infty}f_i(y^{k+}_i)]\subset [f_i(x^-_i),f_i(x^+_i)]$. Similarly, for the case $y^k_i<x_i$, we also can get that result. Then $\cap_{i=1}^n\calF[f_i](x_i)\neq \emptyset$, i.e., $x\in\calD_1$.
\end{proof}

\begin{theorem}\label{th:SC_nonlinear1}
	Suppose the underlying topology $\calG$ is directed and strongly connected, then all the Filippov solutions of \eqref{e:nonlinear1_fili} converge in to $\calD_1$ asymptotically.
\end{theorem}

\begin{proof}
	
	Consider the Lyapunov function $V_1(x)=w^TF(x)$ where $w\in\R^n_{+}$ is given by Property \ref{connected_laplacian} and $F(x)=[F_1(x_1),\ldots,F_n(x_n)]$ with $F_i(x_i)=\int_0^{x_i}f_i(\tau)d\tau$.
	It can be verified that $V_1\in\calC^0$ and $V_1$ is convex which implies that $V_1$ is regular. Moreover, by the monotonicity of $f_i$, we have $\partial F_i(x_i)=[f_i(x_i^-),f_i(x_i^+)]=\calF[f_i](x_i)$. Hence $V_1$ is locally Lipschitz continuous.
	
	Let $\Psi_1$ be defined as
	\begin{equation}
	\Psi_1 = \{t\geq 0 \mid \textnormal{both } \dot{x}(t) \textnormal{ and } \frac{d}{dt}V_1(x(t)) \textnormal{ exist} \}.
	\end{equation}
	Since $x$ is absolutely continuous and $V_1$ is locally Lipschitz, we can let $\Psi_1=\R_{\geq 0}\setminus\bar{\Psi}_1$ where $\bar{\Psi}_1$ is a Lebesgue measure zero set. By Lemma 1 in \cite{Bacciotti1999}, we have
	\begin{equation}
	\frac{d}{dt}V_1(x(t))\in \tilde{\mathcal{L}}_{\calK_1}V_1(x(t))
	\end{equation}
	for all $t\in\Psi_1$ and hence that the set $\tilde{\mathcal{L}}_{\calK_1}V_1(x(t))$ is nonempty for all $t\in\Psi_1$. For $t\in\bar{\Psi}_1$, we have that $\tilde{\mathcal{L}}_{\calK_1}V_1(x(t))$ is empty, and hence $\max \tilde{\mathcal{L}}_{\calK_1}V_1(x(t))< 0$. In the following, we only consider $t\in\Psi_1$.
	
	The gradient of $V_1$ is given as
	\begin{equation}
	\partial V_1(x) =\co \{ \diag(w)\nu \mid \nu\in \bigtimes_{i=1}^n\calF[f_i](x_i) \}.
	\end{equation}
	Then $\forall a\in \tilde{\mathcal{L}}_{\calK_1}V_1(x(t))$, we have that $\exists u\in\bigtimes_{i=1}^n\calF[f_i](x_i)$ such that
	\begin{equation}
	a = -u^T L^T \diag(w)\nu
	\end{equation}
	for all $\nu \in\bigtimes_{i=1}^n\calF[f_i](x_i)$. A special case is that $\nu=u$, which implies that $a\leq 0$ by Property \ref{connected_laplacian}. Hence we have $\max \tilde{\mathcal{L}}_{\calK}V_1(x(t))\leq 0.$ Moreover, $a=0$ if and only if $\bigtimes_{i=1}^n\calF[f_i](x_i)\cap \spa\{\mathds{1}_n\}\neq\emptyset.$ Hence, by the fact that $\calD_1$ is closed, we have $\overline{\{x\in\R^n\mid 	0\in\tilde{\calL}_{\calK}V_1(x)\}}= \calD_1$.
	By Theorem \ref{chap_preli:thm_stability}, all the Filippov trajectories converges into the largest weakly invariant set containing in $\overline{\{x\in\R^n\mid 	0\in\tilde{\calL}_{\calK}V_1(x)\}}.$ Hence the conclusion holds.
\end{proof}

\begin{theorem}\label{th:ST_nonlinear1}
	Suppose the nonlinear functions in \eqref{e:nonlinear1} can be formulated as  $f(x)=[\bar{f}(x_1),\bar{f}(x_2),\ldots,\bar{f}(x_n)]$ where $\bar{f}$ satisfies Assumption \ref{as:monotone}. Then all the Filippov solutions of \eqref{e:nonlinear1_fili} converge in to
	\begin{equation}\label{e:D_1}
	\calD_2=\{ x\in\R^n \mid \exists a\in\R \textnormal{ s.t. } a\mathds{1}_n\in \bigtimes_{i=1}^n\calF[\bar{f}](x_i)\}
	\end{equation} asymptotically if the underlying graph $\calG$ containing a spanning tree.
\end{theorem}

\begin{proof}
	In this case, the differential inclusion \eqref{e:nonlinear1_fili} can be written as
	\begin{equation}\label{e:nonlinear1_fili_barf}
	\begin{aligned}
	\dot{x}   \in -L \bigtimes_{i=1}^n\calF[\bar{f}](x_i)
	 := \calK_2(x).
	\end{aligned}
	\end{equation}

{\it (i)} We show an observation about the behaviors of the trajectories corresponding to roots. Since the subgraph corresponding to the roots is strongly connected, by Theorem \ref{th:SC_nonlinear1}, all the Filippov solution of \eqref{e:nonlinear1_fili_barf} converge that
\begin{equation}
\{x\mid \exists a \textnormal{ s.t. } a\in\calF[\bar{f}](x_i), \forall i\in\calI_r \}.
\end{equation}
where $\calI_r =\{i\in\calI \mid v_i \text{ is a root of }\calG\}$.
	
{\it (ii)} Consider candidate Lyapunov functions $V$ as given in \eqref{e:VandW}.
Let $x(t)$ be a trajectory of \eqref{e:nonlinear1_fili_barf} and define
\[\alpha(x(t))=\{k\in\calI\mid x_k(t)=V(x(t))\}.\]
Denote $x_i(t)=\overline{x}(t)$ for $i\in\alpha(x(t))$.
The generalized gradient of $V$ is given as [\cite{Clarke1990optimization}, Example 2.2.8]
\begin{equation}
\partial V(x(t))  = \mathrm{co}\{e_k\in\R^n \mid k \in \alpha(x(t)) \}.
\end{equation}

Similar to the proof of Theorem \ref{th:SC_nonlinear1}, we can define $\Psi_2$ and we only consider $t\in\Psi_2$ such that $\tilde{\mathcal{L}}_{\calK_2}V(x(t))$ is nonempty and $\R_{\geq 0}\setminus\Psi_2$ is a Lebesgue measure zero set.
For $t \in \Psi_2$, let $a\in\tilde{\mathcal{L}}_{\calK_2}V(x(t))$.
By definition, there exists a $\nu^a \in \bigtimes_{i=1}^n\calF[\bar{f}](x_i)$ such that $a = (-L\nu^a)^{\top}\cdot \zeta$ for all
$\zeta\in\partial V(x(t))$. Consequently, by choosing $\zeta = e_k$ for $k \in \alpha(x(t))$, we observe that $\nu^a$ satisfies
	\begin{equation} \label{e:nu-alpha}
	   -L_{k,\cdot}	\nu^a  = a \qquad \forall k \in \alpha(x(t)).
	\end{equation}

Next, we want to show that $\max \tilde{\mathcal{L}}_{\calK_2}V(x(t))\leq 0$ for all $t\in\Psi_2$ by considering two possible cases: $\calI_r\nsubseteq\alpha(x(t))$ or
$\calI_r\subseteq\alpha(x(t))$.



If $\calI_r\subset \alpha(x(t))$, there are two subcases. First, $|\calI_r|=1$, i.e., there is only one root, denoted as $v_i$. Then $L_{i,\cdot}=0$, hence $L_{i,\cdot}\nu=0$ for any $\nu\in\bigtimes_{i=1}^n\calF[\bar{f}](x_i)$. By the observation \eqref{e:nu-alpha}, we have $\tilde{\mathcal{L}}_{\calK_2}V(x(t))=\{0 \}$. Second, $|\calI_r|\geq 2$. By the fact that the subgraph spanned by the roots is strongly connected, there exists $w_i>0$ for $i\in\calI_r$ such that $\sum_{i\in\calI_r}w_iL_{i,\cdot}=0_n, $
which implies that
\begin{equation}
\sum_{i\in\calI_r}w_iL_{i,\cdot}\nu=0
\end{equation}
for any $\nu\in\bigtimes_{i=1}^n\calF[\bar{f}](x_i)$. Again, by the observation \eqref{e:nu-alpha}, we have $\tilde{\mathcal{L}}_{\calK_2}V(x(t))=\{0 \}$.

If $\calI_r \nsubseteq \alpha(x(t))$, i.e., there exists $i\in\calI_r\setminus \alpha(x(t))$. We define a subset $\alpha'(\nu)$ as
\begin{equation}
\alpha'(\nu)=\{i\in\alpha(x(t)) \mid \nu_i=\max_{i\in\alpha(x(t))} \nu_i \}
\end{equation}
for any $\nu\in\bigtimes_{i=1}^n\calF[\bar{f}](x_i)$. From Property \ref{p:fili_monotone} (ii),  for any $j\in\alpha'(\nu)$, we know that $\nu_j=\max \nu_i$, thus $L_{j,\cdot}\nu\geq 0$. By the fact that the choice of $\nu$ is arbitrary in $\bigtimes_{i=1}^n\calF[\bar{f}](x_i)$ and the observation \eqref{e:nu-alpha}, we have $\tilde{\mathcal{L}}_{\calK_2}V(x(t))\subset \R_{\leq 0}$. Moreover, denoting
\begin{equation}
\calE_{\alpha(x)}=\{e_{ij}\in\calE\mid j\in\alpha(x) \},
\end{equation}
we shall show that $0\in\tilde{\mathcal{L}}_{\calK_2}V(x)$ if and only if $\exists \nu\in \bigtimes_{i=1}^n\calF[\bar{f}](x_i)$ such that $\nu_i=\nu_j$ for any $e_{ij}\in\calE_{\alpha(x)}$, which is equivalent to  $\calF[\bar{f}](x_i)\cap\calF[\bar{f}](x_j)\neq \emptyset$ for all $e_{ij}\in\calE_{\alpha(x)}$.
The sufficient part is straightforward, in fact we can take $\nu_i=\nu_j=f(\overline{x}^-)$ for any $e_{ij}\in\calE_{\alpha(x)}$. Then $0\in\tilde{\mathcal{L}}_{\calK_2}V(x)$. The necessary part can be proved as follows. Since $0\in\tilde{\mathcal{L}}_{\calK_2}V(x)$, there exists $\nu\in \bigtimes_{i=1}^n\calF[\bar{f}](x_i)$ such that $L_{j,\cdot}\nu= 0$ for any $j\in\alpha(x)$. Then this $\nu$ satisfies that $\alpha'(\nu)=\alpha(x)$. Indeed, if $\alpha'(\nu) \subsetneqq \alpha(x)$, then for any $j\in\alpha'(\nu)$ with $e_{ij}\in\calE$ and $i\notin\alpha'(\nu)$, $L_{j,\cdot}\nu<0$. 
Hence $\alpha'(\nu)=\alpha(x)$. Furthermore, by using the same argument, we have for any $e_{ij}\in\calE$ satisfying $i\notin\alpha(x)$ and $j\in\alpha(x)$, $f(\overline{x}^-)\in\calF[\bar{f}](x_i)$.


{\it (iii)} For the Lyapunov functions $W$ as given in \eqref{e:VandW},
denote $\beta(x(t))=\{i\in\calI\mid x_i(t)=-W(x(t))\}$, $x_i(t)=\underline{x}(t)$ for $i\in\beta(x(t))$, and $\calE_{\beta(x(t))} =\{e_{ij}\in\calE\mid j\in\beta(x(t)) \}$. By using similar computations, we find that $\max \tilde{\mathcal{L}}_{\calK_2}W(x(t))\leq 0$ and $0\in\tilde{\mathcal{L}}_{\calK_2}W(x(t))$ if and only if $\exists \nu\in \bigtimes_{i=1}^n\calF[\bar{f}](x_i)$ such that $\nu_i=\nu_j$ for any $e_{ij}\in\calE_{\beta(x(t))}$, which is equivalent to  $\calF[\bar{f}](x_i)\cap\calF[\bar{f}](x_j)\neq \emptyset$ for all $e_{ij}\in\calE_{\beta(x(t))}$.

{\it (iv)} So far we have that $V(x(t))$ and $W(x(t))$ are not increasing along the trajectories $x(t)$ of the
system \eqref{e:nonlinear1_fili_barf}. Hence, the trajectories are bounded and remain in the set $[\underline{x}(0),\overline{x}(0)]^n$ for all $t\geq0$.
Therefore, for any $N\in\R_+$, the set $S_N=\{x\in\R^n \mid \|x\|_{\infty}\leqslant N\}$ is
strongly invariant for \eqref{e:nonlinear1_fili_barf}.
By Theorem \ref{chap_preli:thm_stability}, we
have that all solutions of \eqref{e:nonlinear1_fili_barf} starting in $S_N$
converge to
the largest weakly invariant set $M$ contained in
\begin{equation}\label{e: 0inLVandLW}
\begin{aligned}
S_N & \cap\overline{\{x\in\R^n:
	0\in\tilde{\mathcal{L}}_{\calK_2}V(x)\}} \\
	& \cap\overline{\{x\in\R^n:
	0\in\tilde{\mathcal{L}}_{\calK_2}W(x)\}}.
\end{aligned}
\end{equation}

{\it (v)} We have proved the asymptotic stability of the system. Next we will prove that the set $\calD_2$ is strongly invariant and for any $x_0\notin\calD_2$, all the solution satisfying $x(0)=x_0$ will converge to $\calD_2$.

We start with the strong invariance of $\calD_2$. Notice that by the monotonicity of $\bar{f}$ we can reformulate $\calD_2$ as
\begin{equation}
\calD_2=\{x\mid \calF[\bar{f}](\underline{x})\cap \calF[\bar{f}](\overline{x})\neq \emptyset \}.
\end{equation}
For any $x_0\in\calD_2$, we have known that any trajectories starting from $x_0$, $V(x(t))$ and $W(x(t))$ are not increasing. Hence $\overline{x}(t)\leq \overline{x}_0$ and $\underline{x}(t)\geq \underline{x}_0$ for all $t\geq 0$ which, by Property \ref{p:fili_monotone}, implies that $\calF[\bar{f}](\underline{x}(t))\cap \calF[\bar{f}](\overline{x}(t))\neq \emptyset$ for all $t$ and $x(t)$ satisfying $x(0)=x_0$. Then $x(t)\in\calD_2$ which implies that $\calD_2$ is strongly invariant.

Next we show that for any $x_0\notin\calD_2$, all the solution satisfying $x(0)=x_0$ will converge to $\calD_2$. We will prove it by contradictions. If not, i.e., there exists $x_0\notin\calD_2$ and one solution $\tilde{x}(t)$ satisfying $\tilde{x}(0)=x_0$ does not converge to $\calD_2$. Since the set $\calD_2$ is strongly invariant, we have $\tilde{x}(t)\notin\calD_2$ for all $t\geq 0.$ Then $\calF[\bar{f}](\underline{\tilde{x}})\cap \calF[\bar{f}](\overline{\tilde{x}})= \emptyset$, where
\begin{equation*}
\begin{aligned}
\overline{\tilde{x}}  = \lim_{t\rightarrow\infty}V(\tilde{x}(t)),~
\underline{\tilde{x}}  = -\lim_{t\rightarrow\infty}W(\tilde{x}(t)).
\end{aligned}
\end{equation*}
Hence there exists a constant $C>0$, such that $d(\calF[\bar{f}](\underline{\tilde{x}}), \calF[\bar{f}](\overline{\tilde{x}}))>C$ where $d(S_1,S_2)=\inf_{y_1\in S_1,y_2\in S_2}d(y_1,y_2)$ is the \emph{distance} between two sets $S_1$ and $S_2$. For any $i,j\in\calI$ with $i\neq j$, there exists a vector $w^{ij}\in\R^n$ such that $w^{ij^\top} L=(e_i-e_j)^T$. For each pair $i,j\in\calI$, we choose one $w^{ij}$ and collect all the $w^{ij}$ for $i,j\in\calI$ in the set $\Omega$. Notice that there are only finite number of vectors in $\Omega$.  Then for any $t, i\in\alpha(\tilde{x}(t))$ and $j\in\beta(\tilde{x}(t))$, we have $\overline{\tilde{x}}(t)\geq\overline{\tilde{x}}$ and $ \underline{\tilde{x}}(t)\leq\underline{\tilde{x}} $. Moreover, since $\tilde{x}(t)$ is uniformly bounded, there exist a constant $\tau$ which does not depend on $t$ such that for any $s\in[t, t+\tau]$
\begin{equation}\label{e:differential_lower_bound}
w(s)^T \dot{x}(s)>\frac{C}{2}.
\end{equation}
where $w:\R\rightarrow\Omega$ is piecewise constant and $w(s)=w^{ij}$ with $i\in\alpha(t),j\in\beta(t)$ for $s\in[t,t+\tau]$. Note that for any $T$, the function $w(s)^T\dot{x}(s)$ is Lebesgue integrable on $[0, T]$, and by \eqref{e:differential_lower_bound} we have
\begin{equation}
\int_{0}^T w(s)^\top \dot{x}(s)ds>\frac{C}{2}T
\end{equation}
which converge to infinity as $T\rightarrow \infty$.
This is a contradiction to the fact that $w(s)$ is globally bounded and for any $T<\infty$ and $i\in\calI$, $\int_{0}^T \dot{x}_i(s)ds$ is bounded. Hence we have for any $x_0\notin\calD_2$, all the solution satisfying $x(0)=x_0$ will converge to $\calD_2$. Here ends the proof.

\end{proof}

\begin{remark}\label{r:positive_nonlinear1}
	From the proof of Theorem \ref{th:ST_nonlinear1}, we know the maximal components of the trajectories of the system \eqref{e:nonlinear1_fili} are not increasing while the minimal ones are not decreasing. Hence \eqref{e:nonlinear1_fili} is a positive system (see e.g., \cite{Rantzer2011}), i.e., with positive initial conditions, the trajectories will be positive for all the time.
\end{remark}

\begin{remark}
	A more general case of the dynamical system \eqref{e:nonlinear1} than Theorem \ref{th:ST_nonlinear1}, namely with different nonlinear functions $f_i$ for each agents and the underlying graph being directed containing a spanning tree,  is still open.
\end{remark}

\section{Systems with Nonlinear Communication}\label{s:communication}

In this section we consider a different scenario from Section \ref{s:measure}, namely instead of nonlinear measurement of the agents states, we consider the scenario that the communication among the agents is effected by some nonlinearities. Specifically, we consider the following nonlinear consensus protocol
\begin{equation}\label{e:nonlinear2}
\begin{aligned}
\dot{x}_i & = -\sum_{j=1}^n a_{ij}g_{ij}(x_i-x_j)
\end{aligned}
\end{equation}
where $g_{ij}:\R\rightarrow\R$ satisfying Assumption \ref{as:monotone}.
We understand the solution of \eqref{e:nonlinear2} in the Filippov sense.

In this section, we consider three cases, namely the connected undirected graph, the ring graph, and the directed graphs being a directed spanning tree.

Firstly, we consider that case that the underlying graph is undirected. In this case, we assume that $g_{ij}(\cdot)$ is odd for all $a_{ij}\neq0$, i.e., $g_{ij}(y)=-g_{ij}(-y)$ and let $m$ denotes the number edges. By a given ordering of the $m$ edges, we can re-denote the edges as $e_1, \ldots, e_m$ and the corresponding weight as $a_1, \ldots, a_m$. From the assumption about $g_{ij}$ being odd,	 we can write the system \eqref{e:nonlinear2} in a vectorized form as follows.
	\begin{equation}
	\begin{aligned}\label{bhx}
	\dot{x}  = -B g(B^\top x)
	 := -B h(x)
	\end{aligned}
	\end{equation}
	where $B$ is the incidence matrix and $g(x)=[a_1g_1(x_1),a_2g_2(x_2),\ldots,a_mg_m(x_m)]$.

\begin{theorem}\label{th:measure_undirected}
	Suppose the underlying graph is a connected undirected graph, the nonlinear functions satisfy Assumption \ref{as:monotone} and are odd, then all the Filippov trajectories of \eqref{e:nonlinear2} asymptotically converge into
	\begin{equation}\label{e:H}
	\calH_1=\{ x\in\R^n \mid  0_m\in \bigtimes_{i=1}^m \calF[g_i](B_{\cdot,i}^\top x) \}.
	\end{equation}
\end{theorem}

\begin{proof}
From (\ref{bhx}) and Property \ref{p:calculus for Filippov}, we know that the Filippov differential inclusion is given as
	\begin{equation}\label{e:nonlinear2_fili_undirected}
	\begin{aligned}
	\dot{x} & \in -B\calF[h](x) \\
	& \subset -B \bigtimes_{i=1}^ma_i\calF[g_i](B^\top_{\cdot, i}x):= \calK_3(x).
	\end{aligned}
	\end{equation}
	
	Consider the Lyapunov function $V_3(x)=\frac{1}{2}x^\top x$ which is smooth, hence $\partial V_3(x(t))  = \{x(t) \}.$  The set-valued Lie derivative $\tilde{\mathcal{L}}_{\calK_3}V_3(x)$ is given as
	\begin{equation}
	\begin{aligned}
	& \tilde{\mathcal{L}}_{\calK_3}V_3(x(t))\\
	 = & \{a\in\R \mid a= -x(t)^\top B  \nu, \nu \in \bigtimes_{i=1}^ma_i\calF[g_i](B^\top_{\cdot, i}x(t))  \}.
	\end{aligned}
	\end{equation}
	In this case $\tilde{\mathcal{L}}_{\calK_3}V_3(x(t))\neq \emptyset$ for all the time.

	By the fact that $g_i$ is monotone and $g_i(0)=0$, we have
	\begin{equation}\label{p:sign-preserving_g_fili}
	\calF[g_i](y_i)\subset \begin{cases}
	\R_{\geq 0} & \textrm{ if } y_i>0,\\
	\R_{\leq 0} & \textrm{ if } y_i<0.
	\end{cases}
	\end{equation}
	Hence, $\nu_i$ and $(B^\top x)_i$ have the same sign for any $\nu \in \bigtimes_{i=1}^ma_i\calF[g_i](B^\top_{\cdot, i}x(t))$ and $i\in\calI$. This implies that $\max \tilde{\mathcal{L}}_{\calK_3}V_3(x)\leq 0$.
	By Theorem \ref{chap_preli:thm_stability}, all solutions of (\ref{e:nonlinear2_fili_undirected}) converge to
	the largest weakly invariant set $M$ contained in
	\begin{equation}
	\overline{\{x\in\R^n:
		0\in\tilde{\mathcal{L}}_{\calK_3}V_3(x)\}}.
	\end{equation}
	Notice that $0\in\tilde{\mathcal{L}}_{\calK_3}V_3(x)$ if and only if $0_m\in \bigtimes_{i=1}^m \calF[g_i](B_{\cdot,i}^\top x)$, and the conclusion holds.	
\end{proof}

Before we present next result, we want to show that the condition $g_{ij}(y)=-g_{ij}(-y)$ is a necessary condition to guarantee the boundedness of trajectories.

\begin{figure}
	\centering
	\begin{subfigure}[b]{4cm}
		\begin{tikzpicture}
		\tikzstyle{EdgeStyle}    = [thin,double= black,
		double distance = 0.5pt]
		\useasboundingbox (0,0) rectangle (4cm,1.5cm);
		\tikzstyle{VertexStyle} = [shading         = ball,
		ball color      = white!100!white,
		minimum size = 20pt,%
		inner sep       = 1pt,]
		
		\Vertex[style={minimum
			size=0.2cm,shape=circle},LabelOut=false,L=\hbox{$v_1$},x=1cm,y=0.7cm]{1}
		\Vertex[style={minimum
			size=0.2cm,shape=circle},LabelOut=false,L=\hbox{$v_2$},x=3cm,y=0.7cm]{2}
		\draw
		(1) edge[-,>=angle 90,thin,double= black,double distance = 0.5pt]
		node[below]{$e_{1}$} (2);
		\end{tikzpicture}
		\caption{Undirected graph}
		\label{fig:di-2nodesA}
	\end{subfigure}
	~ 
	\begin{subfigure}[b]{4cm}
		\begin{tikzpicture}
		\tikzstyle{EdgeStyle}    = [thin,double= black,
		double distance = 0.5pt]
		\useasboundingbox (0,0) rectangle (4cm,1.5cm);
		\tikzstyle{VertexStyle} = [shading         = ball,
		ball color      = white!100!white,
		minimum size = 20pt,%
		inner sep       = 1pt,]
		\Vertex[style={minimum
			size=0.2cm,shape=circle},LabelOut=false,L=\hbox{$v_1$},x=1cm,y=0.7cm]{v1}
		\Vertex[style={minimum
			size=0.2cm,shape=circle},LabelOut=false,L=\hbox{$v_2$},x=3cm,y=0.7cm]{v2}
		\draw
		(v1) edge[bend right,->,>=angle 90,thin,double= black,double distance = 0.5pt]
		node[below]{$e_{12}$} (v2)
		(v2) edge[bend right,->,>=angle 90,thin,double= black,double distance = 0.5pt]
		node[above]{$e_{21}$} (v1);
		\end{tikzpicture}
		\caption{Directed ring}
		\label{fig:di-2nodesB}
	\end{subfigure}%
	~ 

	\caption{Two digraphs with two nodes for Examples \ref{ex:2node_undi_sliding} and Remark \ref{r:ring}.}
	\label{fig:di-2nodes}
\end{figure}
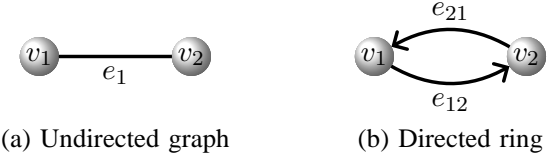

\begin{example}\label{ex:2node_undi_sliding}
	Consider the system \eqref{e:nonlinear2} defined on the undirected graph given as in Fig. \ref{fig:di-2nodesA}.  Furthermore we assume the nonlinear function  $g_1=\varphi$ which is defined as
	\begin{equation}\label{e:varphi}
	\varphi(x)  = \begin{cases}
	1 & \textrm{ if } x>0,\\
	0 & \textrm{ if } x\leq 0,
	\end{cases}
	\end{equation}
	Now the dynamical system can be written as
	\begin{equation}
	\begin{aligned}
	\dot{x}_1 & = \varphi(x_2-x_1) \\
	\dot{x}_2 & = \varphi(x_1-x_2).
	\end{aligned}
	\end{equation}
	With a slight abuse of the notation, we denote
	\begin{equation}
	\varphi(-Lx):=
	\begin{bmatrix}
	\varphi(x_2-x_1) \\
	\varphi(x_1-x_2)
	\end{bmatrix}
	\end{equation}
	where $L$ is the Laplacian matrix of the graph.
	Notice that since $\varphi$ is not an odd function, the previous dynamical system can \emph{not} be written in the form of \eqref{bhx}. Moreover, for any $x_0\in\spa\{\mathds{1}_2 \}$, the Filippov set-valued map
	\begin{equation}\label{e:fili_varphi}
	\calF[\varphi(-Lx)]
	=\overline{\mathrm{co}}\{\begin{bmatrix}
	1 \\ 0
	\end{bmatrix},
	\begin{bmatrix}
	0 \\ 1
	\end{bmatrix} \},
	\end{equation}
	which implies that  $x(t)=x_0+\frac{1}{2}\mathds{1}_2t$ is a Filippov solution. Hence the trajectories can be unbounded. The same conclusion holds for $-\varphi$.
\end{example}

The undesirable behavior $x(t) = \eta(t)\mathds{1}_2$ in the previous example is called \emph{sliding consensus}.

\begin{remark}
	Theorem \ref{th:measure_undirected} is different from Theorem 14 in \cite{Wei2015} in the sense that the sign-preserving (Definition 1 in \cite{Wei2015}) is not assumed for the functions $g_i$ here. Hence, the precise consensus can not be expected in this study.
\end{remark}

Secondly, we consider the case that the underlying graph is a directed ring. Similarly to the undirected case, by relabeling the edges, the dynamical system \eqref{e:nonlinear2} can be written in the following vectorized form
	\begin{equation}\label{e:ring_vec}
	\dot{x}  = g(B^\top x)
	\end{equation}
where $B$ is the incidence matrix of the ring and $g(x)=[a_1g_1(x_1),a_2g_2(x_2),\ldots,a_ng_n(x_n)]$.

\begin{theorem}\label{th:measure_ring}
	Suppose the underlying graph is a ring and all the nonlinear functions $g_{ij}$ satisfy Assumption \ref{as:monotone}. Then all the Filippov trajectories of \eqref{e:nonlinear2} asymptotically converge to
	\begin{equation}
	\calH_2=\{ x\in\R^n \mid  0_n\in \bigtimes_{i=1}^{n} \calF[g_i](B^\top_{\cdot,i}x) \}
	\end{equation}
	if
	\begin{enumerate}
		\item $|\calI|=2$ and $g_{i}$ is odd for any $e_{i}\in\calE$, or \\
		\item $|\calI|\geq3$ and $g_{i}(0)=0, \forall e_{i}\in\calE$ and there exist $e_{i}\in\calE$ such that $\calF[g_{i}](0)=\{0\}$.
	\end{enumerate}
\end{theorem}

\begin{proof}
	By the vectorized form \eqref{e:ring_vec}, the Filippov differential inclusion of \eqref{e:nonlinear2} is given as
	\begin{equation}\label{e:ring_diff_inclusion}
	\dot{x}  \in \calF[g(B^\top x)](x) := \calK_4(x).
	\end{equation}
	
	Since $-B^\top$ is the Laplacian matrix of the reversed ring graph which is also a directed ring, then by Theorem 7 in \cite{Wei2015}, we have that the system \eqref{e:ring_diff_inclusion} is asymptotically stable. More precisely, by the fact that $g_i$ is monotone and $g_i(0)=0$, we have \eqref{p:sign-preserving_g_fili} holds. Furthermore, for any $x\in\spa\{\mathds{1}_n \}$, the Filippov set-valued map $\calF[g(B^\top x)](x)$ satisfies that
	\begin{enumerate}
		\item if $|\calI|=2$ and $g_{i}$ is odd for any $e_{i}\in\calE$,
		\begin{equation}
		\calF[g(B^\top x)](x)=\overline{\mathrm{co}}\{ \begin{bmatrix}
		a_1g_1(0^+) \\ a_2g_2(0^-)
		\end{bmatrix}, \begin{bmatrix}
		a_1g_1(0^-) \\ a_2g_2(0^+)
		\end{bmatrix} \}.
		\end{equation}
		By the fact that $g_i$ is odd, the set $\calF[g(B^\top x)](x)\cap\spa\{\mathds{1}_2 \}=[0,0]^\top$. \\
		\item if $|\calI|>3$ and $g_{i}(0)=0, \forall e_{i}\in\calE$ and there exist $e_{i}\in\calE$ such that $\calF[g_{i}](0)=\{0\}$, w.l.o.g., assume $\calF[g_{1}](0)=\{0\}$. For any $x\in\spa\{\mathds{1}_n \}$, we have $\nu_1=0$ for any $\nu\in \calF[g(B^\top x)](x)$.
	\end{enumerate}
	Then using the similar argument as in the proof of Theorem 7 in \cite{Wei2015}, we have that $\max \tilde{\mathcal{L}}_{\calK_4}V(x(t))\leq 0$ and $\max \tilde{\mathcal{L}}_{\calK_4}W(x(t))\leq 0$ where $V$ and $W$ are given as in \eqref{e:VandW}. This implies that the system \eqref{e:ring_diff_inclusion} is asymptotically stable. Notice that in this paper we do not assume the nonlinear functions to be \emph{sign-preserving} as defined in Definition 1 in \cite{Wei2015}, the exact consensus can not be expected. Next we shall show to which set the trajectories converge.
	
	Consider the coordination transformation $z=B^\top x$. By Property \ref{p:calculus for Filippov}, we have that
	\begin{equation}
	\begin{aligned}
	\dot{z} & = B^\top \dot{x} \\
	& \subset B^\top \calF[g(B^\top x)](x) \\
	& \subset B^\top \bigtimes_{i=1}^{n} a_i\calF[g_i](B^\top_{\cdot,i}x)\\
	& = B^\top \bigtimes_{i=1}^{n} a_i\calF[g_i](z_i).
	\end{aligned}
	\end{equation}
	Again since $-B^\top$ is the Laplacian matrix of the reversed ring graph, we have that the differential inclusion of $z$ is the same as \eqref{e:nonlinear1_fili}. Hence, by  Theorem \ref{th:SC_nonlinear1}, the trajectories $z(t)$ converge to
	$\{ z\in\R^n \mid \exists c\in\R \textnormal{ s.t. } c\mathds{1}\in \bigtimes_{i=1}^na_i\calF[g_i](z_i)\}$. Moreover, by the fact that $\mathds{1}^\top z=0$ and \eqref{p:sign-preserving_g_fili}, we have $c=0$. This implies that the trajectories $x(t)$ of \eqref{e:ring_diff_inclusion} converge to $\calH_2$.
\end{proof}	

\begin{remark}\label{r:ring}
	For the condition $1)$ in Theorem \ref{th:measure_ring}, Example \ref{ex:2node_undi_sliding} can be also employed to show the necessity of have odd function $g_i$. In other words, if $\calI=2$ but $g_1=g_2=\varphi$ is not odd, the sliding consensus could happen. For the condition $2)$, Example 16 in \cite{Wei2015}, which consider the case $g_i=\sign, \forall e_i\in\calE$, shows the necessity of existence $e_i\in\calE$ s.t. $\calF[g_i](0)=\{0\}$ to eliminate the sliding consensus.
\end{remark}

\begin{corollary}\label{th:measure_tree}
	Consider the dynamical system \eqref{e:nonlinear2} defined on a directed spanning tree with $g_{ij}=\bar{g}, \forall e_{ij}\in\calE$  satisfying Assumption \ref{as:monotone} and $\bar{g}(0)=0$. Then all the Filippov trajectories asymptotically converge to
	\begin{equation}
	\begin{aligned}
	\calH_3 =& \{ x\in\R^n \mid \exists \alpha\in\calF[\bar{g}](0) \textnormal{ s.t. }  \\
	&\alpha\mathds{1}_{n-1}\in \bigtimes_{i=1}^{n-1} a_i\calF[\bar{g}](B^\top_{\cdot,i}x) \}.
	\end{aligned}
	\end{equation}
\end{corollary}

\begin{proof}
	Since the underlying graph is a directed spanning tree with the root being denoted as $v_1$, then by Property \ref{p:calculus for Filippov}, the differential inclusion satisfies that
	\begin{equation}
	\begin{aligned}\label{e:tree_diff_inclusion}
	\begin{bmatrix}
	\dot{x}_1 \\ \dot{x}_2 \\ \vdots \\ \dot{x}_n
	\end{bmatrix}
	& \in \begin{bmatrix}
	0 \\  \calF[\bar{g}(B^\top x)](x)
	\end{bmatrix} \\
	& := \calK_5(x).
	\end{aligned}
	\end{equation}
	
	Since the Laplacian matrix of the tree is given as  $L=[0_n, -B]^\top$, it can be verified by \eqref{e:def_filippovset} that
	\begin{equation}
	\calK_5(x)= \calF[\bar{g}(-Lx)](x(t)).
	\end{equation}
	Then by Theorem 7(ii) in \cite{Wei2015}, we have that the system \eqref{e:tree_diff_inclusion} is asymptotically stable.
	This implies that the system \eqref{e:tree_diff_inclusion} is asymptotically stable. Next we shall show to which set the trajectories converge.

	Consider the new coordination $z=[0, B^Tx]$ which satisfies following differential inclusion
	\begin{equation}
	\begin{aligned}
	\dot{z} & \in
	\begin{bmatrix}
	0 \\ B^\top \big(\{0\} \bigtimes \bigtimes_{i=1}^{n-1} a_i\calF[\bar{g}](B^\top_{\cdot,i}x) \big)
	\end{bmatrix}\\
	& =
	\begin{bmatrix}
	0_n \\ B^\top
	\end{bmatrix}
	\{0\} \bigtimes \bigtimes_{i=1}^{n-1} a_i\calF[\bar{g}](B^\top_{\cdot,i}x) \\
	& \subset
	\begin{bmatrix}
	0_n \\ B^\top
	\end{bmatrix}
	\calF[\bar{g}](0) \bigtimes \bigtimes_{i=1}^{n-1} a_i\calF[\bar{g}](B^\top_{\cdot,i}x).
	\end{aligned}
	\end{equation}
	Note that the last inclusion is implied by $\{0\}\subset \calF[\bar{g}](0)$ which can be seen from the assumption that $\bar{g}(0)=0$ and $\bar{g}$ is monotone. Moreover, the Laplacian satisfies
	\begin{equation}
	-L=\begin{bmatrix}
	0_n \\ B^T
	\end{bmatrix}
	\end{equation}
	
	So far we have
	\begin{equation}
	\dot{z} \subset
    -L \Big(\calF[\bar{g}](0) \bigtimes \bigtimes_{i=1}^{n-1} a_i\calF[\bar{g}](z_{i+1}) \Big)
	\end{equation}
	which is in the same form as \eqref{e:nonlinear1_fili}. Hence by Theorem \ref{th:ST_nonlinear1}, the conclusion holds.
\end{proof}

\begin{remark}
	For general directed graphs, the trajectories will \emph{not} converge to the set given as in Theorem \ref{th:measure_ring} and Corollary \ref{th:measure_tree}. An example is given in the following section.
\end{remark}

\section{Applications for quantized consensus protocol}\label{s:quantize}

In this section, we shall reinterpret the results in the previous section for the quantizations. There are three types of most considered quantizers, namely the symmetric, asymmetric and logarithmic quantizer defined as
\begin{equation}
	\begin{aligned}
	\mathtt{q}^s(z) & = \Big\lfloor \frac{z}{\Delta}+\frac{1}{2} \Big\rfloor \Delta,\\
	\mathtt{q}^a(z) & = \Big\lfloor \frac{z}{\Delta} \Big\rfloor \Delta,\\
	\mathtt{q}^l(z) & = \begin{cases}
	\sign(z)\exp\Big(\mathtt{q}^s\big(\ln(|z|)\big)\Big) & \textrm{ if } z\neq 0,\\
	0 & \textrm{ if } z=0,
	\end{cases}
	\end{aligned}
\end{equation}
respectively.

There are some properties about these quantizers. First, for the symmetric quantizer $\mathtt{q}^s$ we have: {\it (i)}
$|\mathtt{q}^s(z)-z|\leq \frac{\Delta}{2}$; {\it (ii)} $\mathtt{q}^s(z)=-\mathtt{q}^s(-z)$. Second, for the asymmetric quantizer $\mathtt{q}^a$, the following relation holds: $0\leq z-\mathtt{q}^a(z)\leq \Delta$. Finally, for the logarithmic quantizer $\mathtt{q}^l$, it satisfies that: {\it (i)} $\mathtt{q}^l(z)=-\mathtt{q}^l(-z)$; {\it (ii)} $|\mathtt{q}^l(z)-z|<\big(\exp(\frac{\Delta}{2})-1 \big)|z|$.

\subsection{Quantized state measurement}\label{s:imprecise_measurement}

The linear consensus protocol given as
\begin{equation*}
\dot{x}_i(t) = -\sum_{j\in\calN_i} \alpha_{ij}(x_i(t)-x_j(t))
\end{equation*}
is a rather idealized system in the sense that each agent has exact information about itself and its neighbors. A very natural question is that what would happen if the information is imprecise for each agent. Specifically, in this subsection we consider the case that the measurement of states of the agents are quantized. More precisely, we consider the following dynamics for agent $i$
\begin{equation}\label{e:quantize_general}
\dot{x}_i=\sum_{j=1}^n a_{ij} \big(q_j(x_j)-q_i(x_i)\big)
\end{equation}
where $q_i:\R\rightarrow\R, i=1,\ldots,n$ a quantizer. If $x\in\R^n$, we denote with some abuse of notation $q(x)=(q_1(x_1),\ldots,q_n(x_n)^T$. Hence the dynamics \eqref{e:quantize_general} can be written in the vector form as
\begin{equation}\label{e:quantizedconsensus_general_1}
\dot{x} = -Lq(x).
\end{equation}

For the case of directed graphs, we consider the quantizers satisfy that $q_i=\mathtt{q}^s, \forall i\in\calI$ and the system \eqref{e:quantizedconsensus_general_1} can be written as
\begin{equation}\label{e:quantizedconsensus1}
\dot{x} = -Lq^s(x).
\end{equation}
In this case the set $\calD_2$ defined as \eqref{e:D_1} is given as
\begin{equation}
\{x\in\R^n \mid \exists k\in\Z \textnormal{ such that } k\Delta\mathds{1}_n\in\mathcal{F}[q^s](x)\},
\end{equation}
which is equivalent to
\begin{align}
\calD_2:=& \{x\in\R^n \mid \exists k\in\Z \textnormal{ s. t. } \\& (k-\frac{1}{2})\Delta \leq x_i\nonumber \leq (k+\frac{1}{2})\Delta, \forall i \in\calI \}.
\end{align}
It is known that without the precise measurement of the states, exact consensus can not be achieved in principle. Instead, the notation of \emph{practical consensus} will be employed. We say that the state variables of the agents converge to \emph{practical consensus}, if $x(t)\rightarrow\calD_2$ as $t\rightarrow\infty$.

Based on Theorem \ref{th:ST_nonlinear1}, we have the following results which is an extension of the result in Section 3 of \cite{Ceragioli2011}. More precisely, we generalize the result in \cite{Ceragioli2011} about undirected graph to the directed one containing a spanning tree.

\begin{corollary}\label{c:quantized_measure}
    Consider the system \eqref{e:quantizedconsensus1} defined on a directed graph containing a spanning tree, all the Filippov solution converge to
    $\calD_2$ asymptotically.
\end{corollary}

\begin{remark}
	By Proposition 1 in \cite{Ceragioliphdthesis}, the Krasovskii and Filippov solutions of \eqref{e:quantizedconsensus1} are equivalent. Hence Corollary \ref{c:quantized_measure} holds for all Krasovskii solution as well.
\end{remark}

\subsection{Communication quantization}\label{s:communication_quantization}


As analogous to the system \eqref{e:quantize_general}, the other scenario is that the communication is imprecise. In particular, we consider the consensus protocol with communication quantization which is given as
\begin{equation}\label{e:quantize_communication}
\dot{x}_i=\sum_{j=1}^n a_{ij} q(x_j-x_i)
\end{equation}
where $q$ is quantizer.

When we specify the quantizer $q$ to be symmetric quantizer $q^s$, we have the set $\calH_1$ defined as in \eqref{e:H} can be expressed as
\begin{equation}\label{e:Hq}
\calH_1:=\{x\in\R^n \mid -\frac{1}{2}\Delta \leq x_i-x_j \leq \frac{1}{2}\Delta, \forall e_{ij}\in\calE \}.
\end{equation}

Then Theorem \ref{th:measure_undirected}, \ref{th:measure_ring} and Corollary \ref{th:measure_tree} can be rewritten as follows.

\begin{theorem}\label{t:commu_quantize_ditree}
Consider the system \eqref{e:quantize_communication} with symmetric quantizer $q^s$, then all the Filippov solutions asymptotically converge into the set $\calH_1$ if
\begin{enumerate}
	\item $\calG$ is undirected, or
	\item $\calG$ is a directed ring, or a directed spanning tree.
\end{enumerate}
\end{theorem}

\begin{proof}
	This theorem is a direct application of the results in Section \ref{s:measure}, since $q^s$ is odd and continuous at the origin which implies that $\calF[q^s](0)=\{0\}$.
\end{proof}

\begin{remark}
	In Theorem \ref{t:commu_quantize_ditree}, the undirected graph case has already been presented in \cite{Guo2013}. In this theorem, we extend that result to the directed graph. Moreover, in the following example, we show that the extension can \emph{not} be made to more general directed graphs.
\end{remark}

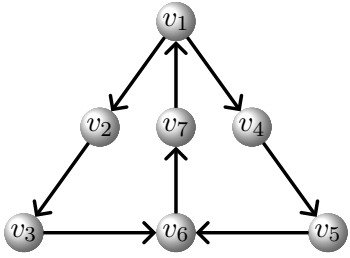
\begin{figure}
	\centering
		\begin{tikzpicture}
		\tikzstyle{EdgeStyle}    = [thin,double= black,
		double distance = 0.5pt]
		\useasboundingbox (0,0) rectangle (6cm,3.5cm);
		\tikzstyle{VertexStyle} = [shading         = ball,
		ball color      = white!100!white,
		minimum size = 20pt,%
		inner sep       = 1pt,]
		
		\Vertex[style={minimum
			size=0.2cm,shape=circle},LabelOut=false,L=\hbox{$v_1$},x=3cm,y=3.4cm]{1}
		\Vertex[style={minimum
			size=0.2cm,shape=circle},LabelOut=false,L=\hbox{$v_2$},x=2cm,y=2.0cm]{2}
		\Vertex[style={minimum
			size=0.2cm,shape=circle},LabelOut=false,L=\hbox{$v_3$},x=1cm,y=0.6cm]{3}
		\Vertex[style={minimum
			size=0.2cm,shape=circle},LabelOut=false,L=\hbox{$v_4$},x=4cm,y=2cm]{4}
		\Vertex[style={minimum
			size=0.2cm,shape=circle},LabelOut=false,L=\hbox{$v_5$},x=5cm,y=0.6cm]{5}						
		\Vertex[style={minimum
			size=0.2cm,shape=circle},LabelOut=false,L=\hbox{$v_6$},x=3cm,y=0.6cm]{6}
		\Vertex[style={minimum
			size=0.2cm,shape=circle},LabelOut=false,L=\hbox{$v_7$},x=3cm,y=2cm]{7}
		\draw
		(1) edge[->,>=angle 90,thin,double= black,double distance = 0.5pt]
		 (2)
		(2) edge[->,>=angle 90,thin,double= black,double distance = 0.5pt]
		 (3)
		(1) edge[->,>=angle 90,thin,double= black,double distance = 0.5pt]
		 (4)
		(4) edge[->,>=angle 90,thin,double= black,double distance = 0.5pt]
		 (5)						
		(3) edge[->,>=angle 90,thin,double= black,double distance = 0.5pt]
		 (6)
		(5) edge[->,>=angle 90,thin,double= black,double distance = 0.5pt]
		 (6)
		(6) edge[->,>=angle 90,thin,double= black,double distance = 0.5pt]
		 (7)
		(7) edge[->,>=angle 90,thin,double= black,double distance = 0.5pt]
		 (1);				
		\end{tikzpicture}
	\caption{Strongly connected digraph used in Examples \ref{ex:general_digraphs}.}
	\label{fig:general_digraph}
\end{figure}

\begin{example}\label{ex:general_digraphs}
	Consider the dynamical system \eqref{e:nonlinear2} defined on a digraph given as in Fig. \ref{fig:general_digraph}. Furthermore we assume the nonlinear function $g_{ij}=q^s$ with quantizer constant $\Delta=1$. Given the initial condition of the state as $x_0=[0, -\frac{1}{3},-\frac{2}{3},\frac{1}{3},\frac{2}{3},0,0]^\top$, then it can be verified that $x(t)=x_0, \forall t>0$ is one Filippov solution. However this solution does \emph{not} belong to the set $\calH_1$ in (\ref{e:Hq}). In fact, $|x_3-x_6|=|x_5-x_6|>\frac{1}{2}\Delta$.
\end{example}

If the quantizer in the system \eqref{e:quantize_communication} is replaced the asymmetric one, i.e., $q^a$, the undesired \emph{sliding consensus} will appear which leads to unboundedness of the trajectories.

\begin{example}
	Consider the dynamical system \eqref{e:quantize_communication} with asymmetric quantizer $q^a$ defined on the graph given as in Fig. \ref{fig:di-2nodesA} and \ref{fig:di-2nodesB}. Since $\calF[q^a](0)=\calF[\varphi](0)$ where $\varphi$ is defined in \eqref{e:varphi}, for any $x\in\spa\{\mathds{1}_2 \}$, the Filippov set-valued map $\calF[q^a(-Lx)](x)=\calF[\varphi(-Lx)](x)$ where $L$ is the Laplacian of the graphs in Fig. \ref{fig:di-2nodes}, and $\calF[\varphi(-Lx)](x)$ is given as \eqref{e:fili_varphi}. Hence, for any $x_0\in\spa \{\mathds{1}_2 \}$, $x(t)=x_0+\frac{1}{2}\mathds{1}t$ is a Filippov solution, i.e., the sliding consensus is a solution.
\end{example}

\section{Conclusion}

In this paper, we considered two general nonlinear consensus protocols, namely the multi-agent systems with nonlinear measurement and communication of their states, respectively. Here we assume the nonlinear functions to be monotonic increasing without any continuity constraints. The solutions of the dynamical systems are understood in the sense of Filippov. For both cases, we proved the asymptotic stability of the systems defined on different topologies. More precisely, in Section \ref{s:measure}, for the case with nonlinear measurement, we considered the systems defined on undirected graphs and directed ones which contain a spanning tree, respectively; in Section \ref{s:communication}, for the case with nonlinear communication, we considered the underlying graph being as undirected, directed ring and directed spanning tree, respectively. Furthermore, we show for the nonlinear communication case, the result can not be extended to general directed graph by examples. Finally, we reinterpret the results in Section \ref{s:measure} and \ref{s:communication} for the quantized consensus protocols, which extend some existing results (e.g., \cite{Ceragioli2011}, \cite{Guo2013}) from undirected graphs to directed ones.


\bibliographystyle{IEEEtran} 
\bibliography{ifacconf}

\begin{thebibliography}{10}
\providecommand{\url}[1]{#1}
\csname url@rmstyle\endcsname
\providecommand{\newblock}{\relax}
\providecommand{\bibinfo}[2]{#2}
\providecommand\BIBentrySTDinterwordspacing{\spaceskip=0pt\relax}
\providecommand\BIBentryALTinterwordstretchfactor{4}
\providecommand\BIBentryALTinterwordspacing{\spaceskip=\fontdimen2\font plus
\BIBentryALTinterwordstretchfactor\fontdimen3\font minus
  \fontdimen4\font\relax}
\providecommand\BIBforeignlanguage[2]{{%
\expandafter\ifx\csname l@#1\endcsname\relax
\typeout{** WARNING: IEEEtran.bst: No hyphenation pattern has been}%
\typeout{** loaded for the language `#1'. Using the pattern for}%
\typeout{** the default language instead.}%
\else
\language=\csname l@#1\endcsname
\fi
#2}}

\bibitem{Ceragioli2011}
F.~Ceragioli, C.~D. Persis, and P.~Frasca, ``Discontinuities and hysteresis in
  quantized average consensus,'' \emph{Automatica}, vol.~47, no.~9, pp. 1916 --
  1928, 2011.

\bibitem{Guo2013}
M.~Guo and D.~V. Dimarogonas, ``Consensus with quantized relative state
  measurements,'' \emph{Automatica}, vol.~49, no.~8, pp. 2531 -- 2537, 2013.

\bibitem{Frasca2012}
P.~Frasca, ``Continuous-time quantized consensus: Convergence of krasovskii
  solutions,'' \emph{Systems \& Control Letters}, vol.~61, no.~2, pp. 273 --
  278, 2012.

\bibitem{Chen2013}
W.~Chen, X.~Li, and L.~Jiao, ``Quantized consensus of second-order
  continuous-time multi-agent systems with a directed topology via sampled
  data,'' \emph{Automatica}, vol.~49, no.~7, pp. 2236 -- 2242, 2013.

\bibitem{Ceragioli2015}
F.~Ceragioli and P.~Frasca, ``Continuous-time consensus dynamics with quantized
  all-to-all communication,'' in \emph{2015 European Control Conference (ECC)},
  July 2015, pp. 1926--1931.

\bibitem{Dimos2010}
D.~V. Dimarogonas and K.~H. Johansson, ``Stability analysis for multi-agent
  systems using the incidence matrix: Quantized communication and formation
  control,'' \emph{Automatica}, vol.~46, no.~4, pp. 695 -- 700, 2010.

\bibitem{Kashyap2007}
A.~Kashyap, T.~Başar, and R.~Srikant, ``Quantized consensus,''
  \emph{Automatica}, vol.~43, no.~7, pp. 1192 -- 1203, 2007.

\bibitem{YiCDC2016}
X.~L. {Yi}, J.~Q. {Wei}, and K.~H. {Johanson}, ``{Self-Triggered Control for
  Multi-Agent Systems with Quantized Communication or Sensing},'' \emph{ArXiv
  e-prints}, Mar. 2016.

\bibitem{Wei2015}
J.~{Wei}, A.~R.~F. {Everts}, M.~K. {Camlibel}, and A.~J. {van der Schaft},
  ``{Consensus problems with arbitrary sign-preserving nonlinearities},''
  \emph{ArXiv e-prints}, Aug. 2015.

\bibitem{Lu2007}
W.~Lu and T.~Chen, ``A new approach to synchronization analysis of linearly
  coupled map lattices,'' \emph{Chinese Annals of Mathematics, Series B},
  vol.~28, no.~2, pp. 149--160, 2007.

\bibitem{cortes2008}
J.~Cortes, ``Discontinuous dynamical systems,'' \emph{Control Systems, IEEE},
  vol.~28, no.~3, pp. 36--73, 2008.

\bibitem{paden1987}
B.~Paden and S.~Sastry, ``A calculus for computing filippov's differential
  inclusion with application to the variable structure control of robot
  manipulators,'' \emph{IEEE Transactions on Circuits and Systems}, vol.~34,
  no.~1, pp. 73--82, 1987.

\bibitem{filippov1988}
A.~Filippov and F.~Arscott, \emph{Differential Equations with Discontinuous
  Righthand Sides: Control Systems}, ser. Mathematics and its
  Applications.\hskip 1em plus 0.5em minus 0.4em\relax Springer, 1988.

\bibitem{Clarke1990optimization}
F.~H. Clarke, \emph{Optimization and Nonsmooth Analysis}, ser. Classics in
  Applied Mathematics.\hskip 1em plus 0.5em minus 0.4em\relax Society for
  Industrial and Applied Mathematics, 1990.

\bibitem{Bacciotti1999}
{A. Bacciotti } and {F. Ceragioli}, ``Stability and stabilization of
  discontinuous systems and nonsmooth lyapunov functions,'' \emph{ESAIM:
  Control, Optimisation and Calculus of Variations}, vol.~4, pp. 361--376,
  1999.

\bibitem{Rantzer2011}
A.~Rantzer, ``Distributed control of positive systems,'' in \emph{2011 50th
  IEEE Conference on Decision and Control and European Control Conference
  (CDC-ECC)}, 2011, pp. 6608--6611.

\bibitem{Ceragioliphdthesis}
F.~Ceragioli, ``Discontinuous ordinary differential equations and
  stabilization,'' Ph.D. dissertation, Universit\`{a} di Firenze, 2000.

\end{thebibliography}

\end{document}